%% file: ex_article.tex
\documentclass[onefignum,onetabnum]{siamart190516}

\input{ex_shared}

\ifpdf
\hypersetup{
  pdftitle={Differential geometry with extreme eiegenvalues},
  pdfauthor={C. Mostajeran, N. Da Costa, G. Van Goffrier, and R. Sepulchre}
}
\fi

\begin{document}

\maketitle

% REQUIRED
\begin{abstract}
Differential geometric approaches to the analysis and processing of data in the form of symmetric positive definite (SPD) matrices have had notable successful applications to numerous fields including computer vision, medical imaging, and machine learning. The dominant geometric paradigm for such applications has consisted of a few Riemannian geometries associated with spectral computations that are costly at high scale and in high dimensions. We present a route to a scalable geometric framework for the analysis and processing of SPD-valued data based on the efficient computation of extreme generalized eigenvalues through the Hilbert and Thompson geometries of the semidefinite cone. We explore a particular geodesic space structure based on Thompson geometry in detail and establish several properties associated with this structure. Furthermore, we define a novel inductive mean of SPD matrices based on this geometry and prove its existence and uniqueness for a given finite collection of points. Finally, we state and prove a number of desirable properties that are satisfied by this mean.
\end{abstract}

% REQUIRED

% REQUIRED

\section{Introduction} \label{sec:introduction}

Geometric data that lie in convex cones appear in a wide variety of applications. Of particular interest is the space of symmetric positive definite (SPD) matrices of a given dimension, which forms the interior of the convex cone of positive semidefinite matrices in the corresponding vector space of symmetric matrices.
In medical imaging, SPD matrices model the covariance matrices of Brownian motion of water in Diffusion Tensor Imaging (DTI) \cite{Pennec2006}. In radar data processing, circular complex random processes with a null mean are characterized by Toeplitz Hermitian positive definite matrices \cite{Arnaudon2013}. In the context of brain-computer interfaces (BCI), where the objective is to enable users to interact with computers via brain activity alone (e.g. to enable communication for severely paralyzed users), the time-correlation of electroencephalogram (EEG) signals are encoded by SPD matrices \cite{Barachant2012}. SPD matrices appear as kernel matrices in machine learning \cite{Lanckriet2004}. SPD representations also find applications in process control, monitoring, and anomaly
detection \cite{Feng2014,Smith2022,Wise1996}, object detection \cite{Tuzel2006,Xu2016}, and the study of functional brain networks \cite{Goni2014,Sporns2002}.

Since SPD matrices do not form a vector space, standard linear analysis techniques applied directly to such data may be inappropriate in some contexts and known to result in poor performance. For instance, the regularization of DTI images using gradient descent algorithms that utilize the classical Euclidean (Frobenius) norm almost inevitably lead to points in the image with negative eigenvalues. Even if we remain in the SPD cone, use of Euclidean (linear) geometry often results in other problems such as `swelling' phenomena in interpolation in DTI \cite{Log-Euclidean2006,Pennec2006} or poor classification results in the context of BCI \cite{Barachant2012,BARACHANT2013,Congedo2017}.

In order to cope with these problems, several Riemannian geometries on SPD matrices have been proposed and used effectively in a variety of applications in computer vision \cite{Huang2018,Jayasumana2015,Minh2017}, medical data analysis \cite{Log-Euclidean2006, Pennec2006, Pennec2020}, machine learning \cite{Cheerian2017,Miolane2020,Zadeh2016}, and optimization \cite{Absil2009, Boumal14, boumal_2023, Mishra2013}. In particular, the affine-invariant Riemannian metric---so-called because it is invariant to affine transformations of the underlying spacial coordinates---has received considerable attention in recent years and applied successfully to problems such as EEG signal processing in BCI where it has been shown to be superior to classical techniques based on feature vector classification \cite{Barachant2012,BARACHANT2013,Congedo2017}. More recently, geometric deep learning architectures have been proposed to learn statistical representations of SPD-valued data that respect the underlying Riemannian geometry \cite{Brooks2019,Huang2017,Ju2022}. The affine-invariant Riemannian geometry has also been applied in the field of geometric statistics where it has been used to construct Riemannian Gaussian distributions, which are used as building blocks for learning models that describe the structure of statistical populations of SPD matrices \cite{Chen2022,Said2017,Said2018,Said2023,HOS2022,Tupker2021}.

The affine-invariant Riemannian metric endows the space of SPD matrices of a given dimension with the structure of a Hadamard manifold with non-constant negative curvature \cite{Lang1999}. Computing standard geometric objects such as distances, geodesics, Riemannian exponentials and logarithms in this geometry often amounts to the computation of the generalized eigenspectrum of a pair of SPD matrices, which typically means a significant increase in computational complexity, particularly for larger matrices. In particular, the algorithms for computing the affine-invariant Riemannian geodesic between two SPD matrices of moderate size, often interpreted as the weighted geometric mean, become unfeasible for large matrices \cite{Iannazo2016}. More recently, there have been successful efforts in developing scalable algorithms for the computation of the product of the weighted geometric mean and a vector, with applications to the domain decomposition preconditioning of PDEs \cite{Arioli2012} and clustering of signed complex networks \cite{Fasi2018,Mercado2016}. While these methods can be highly effective in computing the action of the weighted geometric mean on a vector, they do not typically provide a scalable algorithm for the construction of the full matrix.

An important point that has not received much attention in the literature on geometric optimization and statistics involving SPD-valued data is that there are natural non-Riemannian geometries that can be associated with SPD matrices based on the conic structure of the space. In particular, the Hilbert and Thompson metrics \cite{Baggio2019,Lemmens2012, Mostajeran2020,Nielsen2019,Thompson1963} on the cone of SPD matrices generate non-Euclidean geometries with a rich set of properties including distance and geodesic computations that rely only on extreme generalized eigenvalues \cite{Mostajeran2020,VanGoffrier2021}, which are efficiently computable using techniques such as Krylov subspace methods based on matrix-vector products \cite{Ge2016,Golub2000,Stewart2002,Sun2016}. The full utilization of non-Euclidean geometries that are naturally suited to the SPD cone in the design of cost functions and optimization algorithms for problems involving SPD-valued data offers the potential for enhanced analytic insights and dramatic improvements in computational efficiency over existing costly Riemannian methods. 

\subsection{Hilbert and Thompson geometries}

Let $V$ be a finite-dimensional real vector space. A subset $K$ of $V$ is called a cone if it is convex, $\mu K \subseteq K$ for all $\mu \geq 0$, and $K\cap(-K)=\{0\}$. It is said to be a closed cone if it is a closed set in $V$ with respect to the standard topology. A cone is said to be solid if it has non-empty interior. We say that a cone is almost Archimedean if the closure of its restriction to any two-dimensional subspace is also a cone. Examples of solid closed cones include the positive orthant $\mathbb{R}_{+}^n=\{(x_1,\dotsc,x_n)\in\mathbb{R}^n:x_i\geq 0, \; 1\leq i \leq n\}$ and the set of positive semidefinite matrices in the space of real $n\times n$ matrices.

A cone $K$ in a vector space $V$ induces a partial ordering on $V$ given by $x\leq y$ if and only if $y-x\in K$. For each $x\in K\setminus \{0\}$, $y\in V$, define $M(y/x):=\inf\{\lambda\in\mathbb{R}:y\leq \lambda x\}$. Hilbert's projective metric on $K$ is defined to be
\begin{equation} \label{Hilbert metric}
    d_H(x,y)=\log(M(y/x)M(x/y)).
\end{equation}
Hilbert's projective metric is a pseudo-metric on the cone since it can be shown that $d_H(x,y)=0$ if and only if $x=\lambda y$ for some $\lambda>0$. Indeed, $d_H$ defines a metric on the space of rays of the cone \cite{Lemmens2012}. A specific example of Hilbert geometry is $n$-dimensional hyperbolic space, which is isometric to the the Lorentz cone $\{(t,x_1,\dotsc,x_n)\in\mathbb{R}^{n+1}:t^2>x_1^2+\dotsc+x_n^2\}$ endowed with its Hilbert metric. However, Hilbert geometry only corresponds to a CAT(0) space if the cone is Lorentzian \cite{Bridson2013}. Thus, Hilbert geometry is certainly more general than hyperbolic geometry. Beyond geometry, Hilbert's projective metric finds important applications in analysis, where many naturally arising linear and nonlinear maps are either non-expansive or contractive with respect to it \cite{Birkhoff1957,Bushell1973,Lemmens2012,Sepulchre2010CDC}.

Thompson's part metric on $K$ is a closely related metric that is defined to be
\begin{equation} \label{Thompson metric}
    d_T(x,y)=\log(\max\{M(y/x),M(x/y)\}).
\end{equation}
Two points in $K$ are said to be in the same part if the distance between them is finite in the Thompson metric. If $K$ is almost Archimedean, then each part of $K$ is a complete metric space with respect to the Thompson metric \cite{Thompson1963}. 

Turning our attention to the case of the positive semidefinite cone, we find that for strictly positive definite matrices $X, Y\succ 0$, $M(Y/X)=\lambda_{\max}(YX^{-1})=1/\lambda_{\min}(XY^{-1})$, where $\lambda_{\max}(A)$ and $\lambda_{\min}(A)$ denote the maximum and minimum eigenvalues of the matrix $A$, respectively. Note that $\lambda_{\max}(YX^{-1})$ is well-defined since $YX^{-1}$ is a diagonalizable matrix with real and positive eigenvalues. It follows that the Hilbert and Thompson metrics take the form
\begin{equation} \label{Hilbert matrix}
    d_H(X,Y)=\log\left(\frac{\lambda_{\max}(YX^{-1})}{\lambda_{\min}(YX^{-1})}\right) 
\end{equation}
and
\begin{equation} \label{Thompson matrix}
    d_T(X,Y)=\log\left(\max\{\lambda_{\max}(YX^{-1}),1/\lambda_{\min}(YX^{-1})\}\right).
\end{equation}
%Note that the Hilbert and Thompson metric distances for a given pair of positive definite matrices can be computed using only the extremal generlized eigenvalues of the pair of matrices.

\subsection{Paper organization and contributions}

The main aim of this paper is to provide a connection between the differential geometry of SPD matrices—which has been the subject of significant research interest in recent years accompanied by notable successful applications—and numerical linear algebra, specifically iterative methods for computing extreme eigenvalues—a cornerstone of modern applied mathematics and computing. In this paper, the Hilbert and Thompson geometries of the semidefinite cone are used as a route to establish such a connection. 

In \cref{sec:metrics}, we review affine-invariant metric geometry in the SPD cone and observe how the Thompson metric arises naturally as a member of a family of affine-invariant metrics generated by a collection of Finsler metrics. In \cref{sec:geodesics}, we consider geodesics in Thompson geometry and a choose a particular geodesic with attractive computational properties as a distinguished geodesic whose properties we examine closely. In \cref{sec:inductive mean}, we introduce a novel inductive mean of any finite collection of SPD matrices as the limit of a sequence that is generated through constructions of Thompson geodesics (\cref{alg:inductive mean}) that can be efficiently computed in high dimensions using extreme generalized eigenvalues. We prove that this novel inductive mean of SPD matrices is well-defined by showing that any sequence generated by \cref{alg:inductive mean} converges to a unique point that is independent of the choice of initialization (\cref{cvgence_thm}) and the ordering of the SPD matrices. Furthermore, we state and prove a number of desirable properties that are satisfied by this mean in \cref{properties_thm}.

\section{Affine-invariant metric geometry} \label{sec:metrics}

Let $\mathbb{S}^n_{++}$ denote the space of $n\times n$ real symmetric positive definite matrices. It is well-known that $\mathbb{S}^n_{++}$ admits a Riemannian distance function $d_2:\mathbb{S}^n_{++}\times\mathbb{S}^n_{++}\rightarrow \mathbb{R}$ 
\begin{equation} \label{R distance}
d_2(X,Y)=\left(\sum_{i=1}^n\log^2\lambda_i(YX^{-1})\right)^{1/2},
\end{equation}
where $\lambda_i(YX^{-1})=\lambda_i(X^{-1/2}YX^{-1/2})$ denote the $n$ real and positive eigenvalues of $YX^{-1}$. \cref{R distance} endows $\mathbb{S}_{++}^n$ with the structure of a Riemannian symmetric space and a metric space of nonpositive curvature \cite{HOS2022}. It can be viewed as a Riemannian extension of the logarithmic distance between positive scalars $d(x,y)=|\log(y/x)|$ to positive definite matrices \cite{Bonnabel2010,LimSep2019,Mostajeran2020} and possesses a number of remarkable symmetries that lie behind its utility in a variety of applications including brain-computer interfaces \cite{Barachant2012,BARACHANT2013,Congedo2017,Ju2022}, computer vision \cite{Huang2018},  medical imaging \cite{Log-Euclidean2006,Pennec2006}, radar signal processing \cite{Arnaudon2013}, statistical inference \cite{Said2017,Said2018}, and machine learning \cite{Huang2017,Zadeh2016}. These symmetries include affine-invariance, i.e., invariance under congruence transformations:
$d_2(X,Y)=d_2(AXA^T,AYA^T)$ for any invertible $A\in\mathrm{GL}(n,\mathbb{R})$, where $A^T$ denotes the transpose of $A$ \cite{FLETCHER2007,MostajeranGSI2019,MostajeranGSI2017,Mostajeran2018,Pennec2006,ThanwerdasGSI2019}. Another key symmetry satisfied by this metric is invariance under matrix inversion: $d_2(X,Y)=d_2(X^{-1},Y^{-1})$.

While the Riemannian distance \cref{R distance} has been the subject of significant research interest due to its symmetries and use in applications, it should be noted that it is only one member of a family of distance functions on $\mathbb{S}^n_{++}$ that enjoy the same properties. Indeed, the distances $d_{\Phi}$ on $\mathbb{S}^n_{++}$ defined as 
\begin{equation} \label{distance family}
    d_{\Phi}(X,Y)=\|\log X^{-1/2}YX^{-1/2}\|_{\Phi},
\end{equation}
where $\|\cdot\|_{\Phi}$ is an orthogonally invariant norm on the space of $n\times n$ symmetric matrices given by $\|Z\|_{\Phi}=\Phi(\lambda_1(Z),\cdots,\lambda_n(Z))$, $\lambda_i(Z)$ denote the eigenvalues of $Z$, and $\Phi$ is a symmetric gauge function on $\mathbb{R}^n$, are affine-invariant and inversion-invariant distances \cite{Bhatia2003}. The symmetric gauge functions corresponding to the $l_p$-norms in $\mathbb{R}^n$ induce the Schatten $p$-norms $\|\cdot\|_{\Phi}$ for $1\leq p \leq \infty$. If we take $\Phi(x_1,\cdots,x_n)=(\sum_i x_i^2)^{1/2}$, $d_{\Phi}$ yields the Riemannian distance function \cref{R distance}, whereas the choice of $\Phi(x_1,\cdots,x_n)=\max_i |x_i|$ yields the Thompson metric \cref{Thompson matrix}, which can equivalently be expressed as
\begin{equation} \label{d inf}
    d_{\infty}(X,Y)=\max_{1\leq i \leq n}|\log\lambda_i(YX^{-1})|=\max\{\log\lambda_{\max}(YX^{-1}),\log\lambda_{\max}(XY^{-1})\}.
\end{equation}
The form of the right-hand side of \cref{d inf} is of computational significance since it only involves the computation of the largest generalized eigenvalues of the pairs $(X,Y)$ and $(Y,X)$. Thus, we see that the Thompson metric is both affine-invariant and inversion-invariant. 

The space $\mathbb{S}^n_{++}$ is an open subset of the vector space of $n\times n$ real symmetric matrices and inherits a natural structure of a real differentiable manifold as a result.
From a differential viewpoint, the distance functions $d_{\Phi}$ are induced by affine-invariant Finsler metrics on $\mathbb{S}^n_{++}$ given by the norm $\|d\Sigma\|_{\Sigma,\Phi} := \|\Sigma^{-1/2}d\Sigma\Sigma^{-1/2}\|_{\Phi}$ defined on the tangent space at $\Sigma\in \mathbb{S}^n_{++}$. In particular, the Thompson distance $d_{T}(X,Y)$ is induced by the norm
\begin{equation}
    \|d\Sigma\|_{\Sigma}=\inf\{\alpha>0: -\alpha\Sigma \leq d\Sigma \leq \alpha\Sigma\}
\end{equation}
and is recovered by minimizing the length 
\begin{equation}
    L[\gamma]=\int_0^1\|\gamma'(t)\|_{\gamma(t)}dt
\end{equation}
over all piecewise $C^1$ curves $\gamma:[a,b]\rightarrow \mathbb{S}^n_{++}$ with $\gamma(0)=X$ and $\gamma(1)=Y$ \cite{Nussbaum1994}. The Hilbert metric is recovered through a similar procedure by replacing the above norm with the semi-norm $\|d\Sigma\|_{\Sigma}=M(d\Sigma/\Sigma)-m(d\Sigma/\Sigma)$, where $M(d\Sigma/\Sigma)=\inf\{\lambda\in\mathbb{R}:d\Sigma\leq \lambda \Sigma\}$ and $m(d\Sigma/\Sigma)=\sup\{\lambda\in\mathbb{R}:d\Sigma\geq \lambda \Sigma\}$ \cite{Nussbaum2004}. Various unit balls centered on the identity matrix in these affine-invariant geometries are depicted in \cref{fig:Thompson balls}
in the case of $2\times 2$ SPD matrices visualized as the interior of a convex cone $\{(a,b,c)\in\mathbb{R}^3: a \geq 0, \, ac-b^2 \geq 0\}$.

\begin{figure} 
    \centering
    \includegraphics[width=1\linewidth]{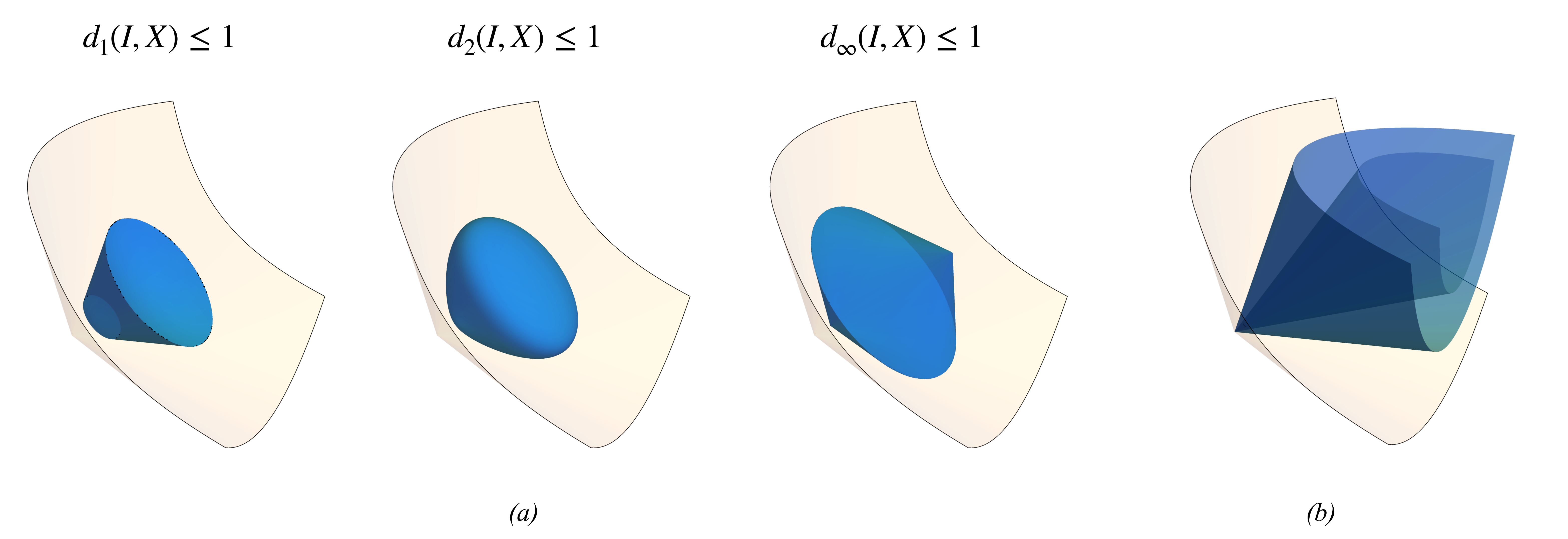}
    \caption{(a) Unit balls $d_{\Phi}(I,X)\leq 1$ in the affine-invariant geometries induced by the gauge functions $\Phi$ corresponding to the $l_1$-, $l_2$-, and $l_{\infty}$-norms in $\mathbb{R}^2$ visualized as points in the interior of the closed convex cone $\{(a,b,c)\in\mathbb{R}^3: a\geq 0, \, ac-b^2 \geq 0\}$, which we identify with the set of $2\times 2$ SPD matrices. Note that $d_{\infty}$ corresponds to the Thompson metric. (b) The sets $d_{H}(I,X)\leq 1/2$ and $d_H(I,X)\leq 1$ in Hilbert's projective metric applied to $2\times 2$ SPD matrices visualized in $\mathbb{R}^3$.}
    \label{fig:Thompson balls}
\end{figure}

\section{Geodesics} \label{sec:geodesics}

A geodesic path in a metric space $(M,d)$ is a map $\gamma:I\rightarrow (M,d)$ such that $d(\gamma(s),\gamma(t))=|s-t|$ for all $s,t\in I$, where $I\subseteq \mathbb{R}$ is a (possibly unbounded) interval. The image of a geodesic path is called a geodesic and a metric space is said to be a geodesic space if there exists a geodesic path joining any two points. Each of the metric spaces $(\mathbb{S}_{++}^n,d_{\Phi})$ with $d_{\Phi}$ defined in \cref{distance family} is a geodesic space. Indeed, the curve $\gamma:[0,1]\rightarrow \mathbb{S}^n_{++}$ defined by
\begin{equation} \label{Riemannian geodesic}
\gamma(t)= X\#_t Y \vcentcolon = X^{1/2}(X^{-1/2}YX^{-1/2})^t X^{1/2}
\end{equation}
is a geodesic path from $X$ to $Y$ in each of these metric spaces and is unique provided that the geodesics in $\mathbb{R}^n$ induced by $\Phi$ are unique \cite{Bhatia2003,Lang1999}. Thus, uniqueness of geodesics in $(\mathbb{S}_{++}^n,d_{\Phi})$ is inherited from $\mathbb{R}^n$ when $\Phi$ corresponds to the $l_p$-norms for $1<p<\infty$, but not for $p=1,\infty$. 

In general, the Thompson metric does not admit unique geodesic paths between points. Indeed, a construction by Nussbaum in \cite{Nussbaum1994} describes a family of geodesics that generally consists of an infinite number of curves connecting a pair of points in a cone $K$. In particular, setting $\alpha:=1/M(x/y;K)$ and $\beta:=M(y/x;K)$, the curve $\phi:[0,1]\rightarrow K$ given by
\begin{equation} \label{Nussbaum geodesic}
\phi(t;x,y)= x*_t y \vcentcolon = \begin{dcases}
\left(\frac{\beta^t-\alpha^t}{\beta-\alpha}\right)y+\left(\frac{\beta\alpha^t-\alpha\beta^t}{\beta-\alpha}\right)x \quad &\mathrm{if} \; \alpha\neq\beta, \\
\alpha^t x &\mathrm{if} \; \alpha=\beta,
\end{dcases}
\end{equation}
is a geodesic path from $x$ to $y$ with respect to the Thompson metric. If we take $K$ to be the cone of positive semidefinite matrices with interior $\operatorname{int} K= \mathbb{S}^n_{++}$, then for a pair of points $X,Y\in\mathbb{S}^n_{++}$, we have $\beta=M(Y/X;K)=\lambda_{\max}(YX^{-1})$ and $\alpha=1/M(X/Y;K)=\lambda_{\min}(YX^{-1})$. Therefore, $X*_t Y$ reduces to a linear combination of $X$ and $Y$ with coefficients that are nonlinear functions of the extreme generalized eigenvalues of $(X,Y)$ and $t$. 

\begin{proposition}\label{affine_invariant_prop}
If $A\in \mathrm{GL}(n)$ and $X,Y\in \mathbb{S}^n_{++}$, then $(AXA^T)*_t(AYA^T)=A(X*_tY)A^T$ for any $t\in \mathbb{R}$.
\end{proposition}
\begin{proof}
The proof follows by noting that $(AYA^T)(AXA^T)^{-1}=AYX^{-1}A^{-1}$ and $YX^{-1}$ have the same eigenvalues and using elementary algebra.
\end{proof}

\begin{proposition}
If $X,Y\in \mathbb{S}^2_{++}$, then $X\#_t Y = X*_t Y$ for all $t\in [0,1]$.
\end{proposition}
\begin{proof}
By the density of dyadic rationals in the real line, it is sufficient to prove that $X\#_{1/2} Y = X*_{1/2} Y$ for arbitrary $X$ and $Y$. Moreover, by affine-invariance and the uniqueness of the Riemannian geodesic, it is sufficient to prove that $I\#_{1/2} \Sigma = I*_{1/2}\Sigma$ for arbitrary $\Sigma \in \mathbb{S}^2_{++}$. This is equivalent to 
\begin{align*}
    \Sigma^{1/2}=\frac{1}{\sqrt{\lambda_{\max}}+\sqrt{\lambda_{\min}}}\left(\Sigma+\sqrt{\lambda_{\max}\lambda_{\min}}\,I\right),
\end{align*}
where $\lambda_i$ denote the eigenvalues of $\Sigma$. However, this equality is seen to hold since $\Sigma^{1/2}$ is a $2\times 2$ matrix with spectrum $\{\sqrt{\lambda_{\min}},\sqrt{\lambda_{\max}}\}$ and characteristic equation $p(\lambda)=\lambda^2-(\sqrt{\lambda_{\max}}+\sqrt{\lambda_{\min}})\lambda + \sqrt{\lambda_{\max}\lambda_{\min}} = 0$, which is of course satisfied by $\Sigma^{1/2}$ by the Cayley-Hamilton theorem.
\end{proof}

In general, of course, the geodesics $X\#_t Y$ and $X*_t Y$ do not agree in higher dimensions. Indeed, the two choices of geodesic agree in $\mathbb{S}^n_{++}$ if and only if the spectrum of $YX^{-1}$ consists of at most two distinct eigenvalues \cite{Lim2013}. It should be noted that even in $\mathbb{S}^2_{++}$ where the $\#_t$ and $*_t$ geodesics agree, the Thompson geodesic is still not unique. Indeed, it is shown in \cite{Lim2013} that there exists a unique Thompson geodesic from $X$ to $Y$ in $\mathbb{S}^n_{++}$ if and only if the spectrum of $YX^{-1}$ is contained in $\{\lambda,\lambda^{-1}\}$ for some fixed $\lambda>0$.
For example, the following construction describes another geodesic $X\diamond_t Y$ of $(\mathbb{S}^n_{++},d_T)$ from $X$ to $Y$ when $\lambda_{\max}(YX^{-1})\neq \lambda_{\min}(YX^{-1})$:
\begin{equation}
    X\diamond_t Y = \begin{dcases}
        \frac{\lambda_{\max}^t-\lambda_{\max}^{-t}}{\lambda_{\max}-\lambda_{\max}^{-1}}\,Y + \frac{\lambda_{\max}^{1-t}-\lambda_{\max}^{t-1}}{\lambda_{\max}-\lambda_{\max}^{-1}}\,X, \quad \lambda_{\max}\lambda_{\min}\geq 1 \\
        \frac{\lambda_{\min}^t-\lambda_{\min}^{-t}}{\lambda_{\min}-\lambda_{\min}^{-1}}\,Y + \frac{\lambda_{\min}^{1-t}-\lambda_{\min}^{t-1}}{\lambda_{\min}-\lambda_{\min}^{-1}}\,X, \quad \lambda_{\max}\lambda_{\min}\leq 1,
    \end{dcases}
\end{equation}
where $\lambda_{\max}$ and $\lambda_{\min}$ refer to the corresponding eigenvalues of $YX^{-1}$ \cite{Lim2013,Nussbaum1994}. A depiction of these various geodesics for an example computed in the set $\mathbb{S}^2_{++}$ visualized as the interior of a cone in $\mathbb{R}^3$ is shown in \cref{fig:2x2 geodesics}. We thus note that $*_t$ is special among the Thompson geodesics constructed by Nussbaum \cite{Nussbaum1994} in that it coincides with the Riemannian geodesic for $2\times 2$ SPD matrices. The $*_t$ geodesic satisfies other desirable properties that do not generally hold for other Thompson geodesics such as joint homogeneity, which is also satisfied by the Riemannian geodesic in all dimensions.

\begin{proposition}[Joint homogeneity]
    Let $X_1, X_2\in \mathbb{S}^n_{++}$. If $\mu_1$ and $\mu_2$ are positive scalars, then
    \begin{equation}
        (\mu_1X_1)*_t(\mu_2X_2)=\mu_1^{1-t}\mu_2^t(X_1*_tX_2)
    \end{equation}
    for any $t\in \mathbb{R}$.
\end{proposition}
\begin{proof}
    The result follows from the equality $\lambda_i\left((\mu_2 X_2(\mu_1 X_1)^{-1}\right)=\frac{\mu_2}{\mu_1}\lambda_i(X_2X_1^{-1})$ and substitution into the expression for $(\mu_1X_1)*_t(\mu_2X_2)$ arising from \cref{Nussbaum geodesic}. 
\end{proof}

\begin{corollary}
    If $X_1, X_2\in \mathbb{S}^n_{++}$, then $(\mu_1X_1)*_{\frac{1}{2}}(\mu_2X_2)=\sqrt{\mu_1\mu_2}\,(X_1*_{\frac{1}{2}}X_2)$ for any positive scalars $\mu_1$ and $\mu_2$.
\end{corollary}

We will view the $*_t$ Thompson geodesic \cref{Nussbaum geodesic} as a distinguished geodesic of $(K,d_T)$, which makes the resulting structure a geodesic space. For the remainder of this paper, by ``Thompson geodesic'' we refer specifically to the $*_t$ geodesic unless stated otherwise.

\begin{figure} 
    \centering
    \includegraphics[width=0.5\linewidth]{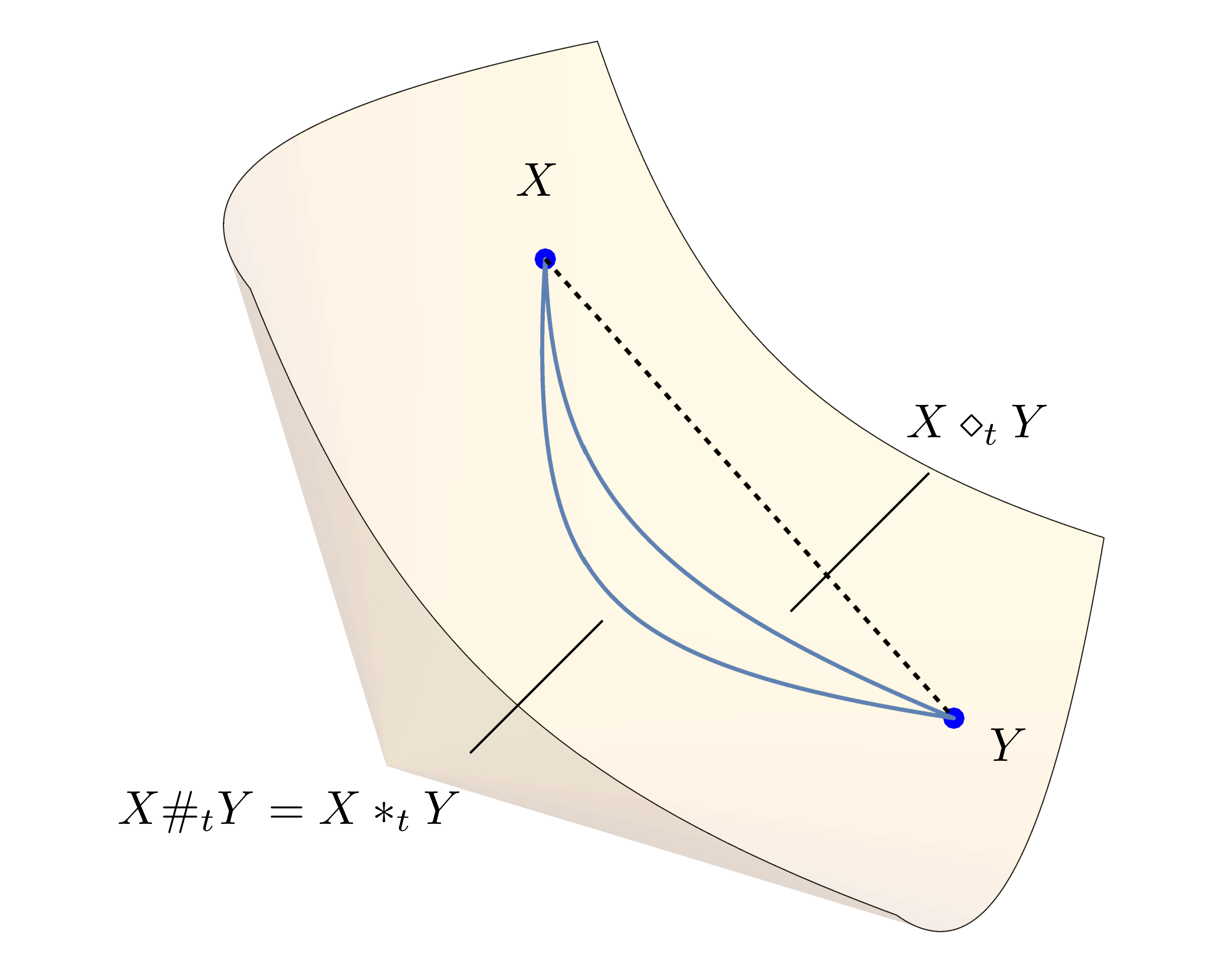}
    \caption{Two geodesics in $(\mathbb{S}^2_{++},d_T)$ between a pair of matrices visualized as points in the interior of the closed convex cone $\{(a,b,c)\in\mathbb{R}^3: a>0, \, ac-b^2>0\}$. The dashed straight line between the endpoints does not represent a geodesic.}
    \label{fig:2x2 geodesics}
\end{figure}

\subsection{Metric inequalities in Hilbert and Thompson geometries}
The following theorem from \cite{Nussbaum2004} establishes two important inequalities in the Thompson and Hilbert geometries of convex cones that provide insight into the curvature properties of these geometries. These inequalities can be viewed as describing how far the Thompson and Hilbert geometries are from being non-positively curved.

\begin{theorem}[Theorems 1.1 and 1.2 of \cite{Nussbaum2004}] \label{Nussbaum inequality}
Let $K$ be an almost Archimedean cone and $u,x,y\in K$ be in the same part of $K$. Suppose that $0 < s < 1$ and $R>0$, and that $d_H(u,x)\leq R$ and $d_H(u,y)\leq R$. If the linear span of $\{u,x,y\}$ is 1- or 2-dimensional, then $d_T(u*_s x, u*_s y)\leq s d_T(x,y)$ and $d_H(u*_s x, u*_s y)\leq s d_H(x,y)$. In general,
\begin{align}
    d_T(u*_s x, u*_s y) &\leq \left[\frac{2(1-e^{-Rs})}{1-e^{-R}}-s\right]d_T(x,y) \\
    d_H(u*_s x, u*_s y) &\leq \left(\frac{1-e^{-Rs}}{1-e^{-R}}\right)d_H(x,y).
\end{align}
\end{theorem}

A remarkable feature of \cref{Nussbaum inequality} is that it ties the Hilbert and Thompson geometries of a convex cone together and suggests that one should consider both of these metrics in geometric analysis in convex cones rather than making a choice of one over the other. A consequence of \cref{Nussbaum inequality} is that both the Hilbert and Thompson geometries are semihyperbolic in the sense of Alonso and Bridson \cite{Alonso1995}. 

\begin{corollary}
$\mathbb{S}^n_{++}$ is semihyperbolic when endowed with Hilbert's projective metric or Thompson's part metric.
\end{corollary}

\subsection{Sparsity preservation}

Sparse matrices are matrices whose non-zero elements form a relatively small proportion of the  matrix entries. They appear in many areas of applied mathematics and engineering including the numerical analysis of partial differential equations, network theory, and machine learning. They arise naturally in multi-agent systems that include relatively few pairwise interactions. From a computational perspective, sparsity is an important property due to the existence of specialized algorithms and data structures that enable the efficient storage and manipulation of large sparse matrices \cite{Gilbert1992}. 

An interesting property of the $*_t$ Thompson geodesic is that it preserves sparsity. That is, if $X$ and $Y$ are sparse SPD matrices, then $X*_t Y$ is sparse for every $t\in \mathbb{R}$. This is simply a consequence of $X*_t Y$ being a linear combination of $X$ and $Y$ for any fixed $t$. In contrast, the Riemannian geodesic $X\#_t Y$, whose construction involves computing matrix square roots, matrix products, and matrix inverses, does not preserve sparsity. Thus, the use of Riemannian interpolation to process large sparse SPD matrices may be problematic. For instance, kernel matrices in machine learning are often built as sparse matrices to facilitate the analysis of large datasets. Applying the standard affine-invariant Riemannian geometry to process such SPD matrices will typically corrupt the sparse structure, potentially resulting in intractable computations. See \cref{fig:sparsity} for a visualization of Riemannian and $*_t$ Thompson geodesic interpolations of a pair of $20\times 20$ SPD matrices with 68 non-zero entries.

\begin{figure} 
    \centering
    \includegraphics[width=1.0\linewidth]{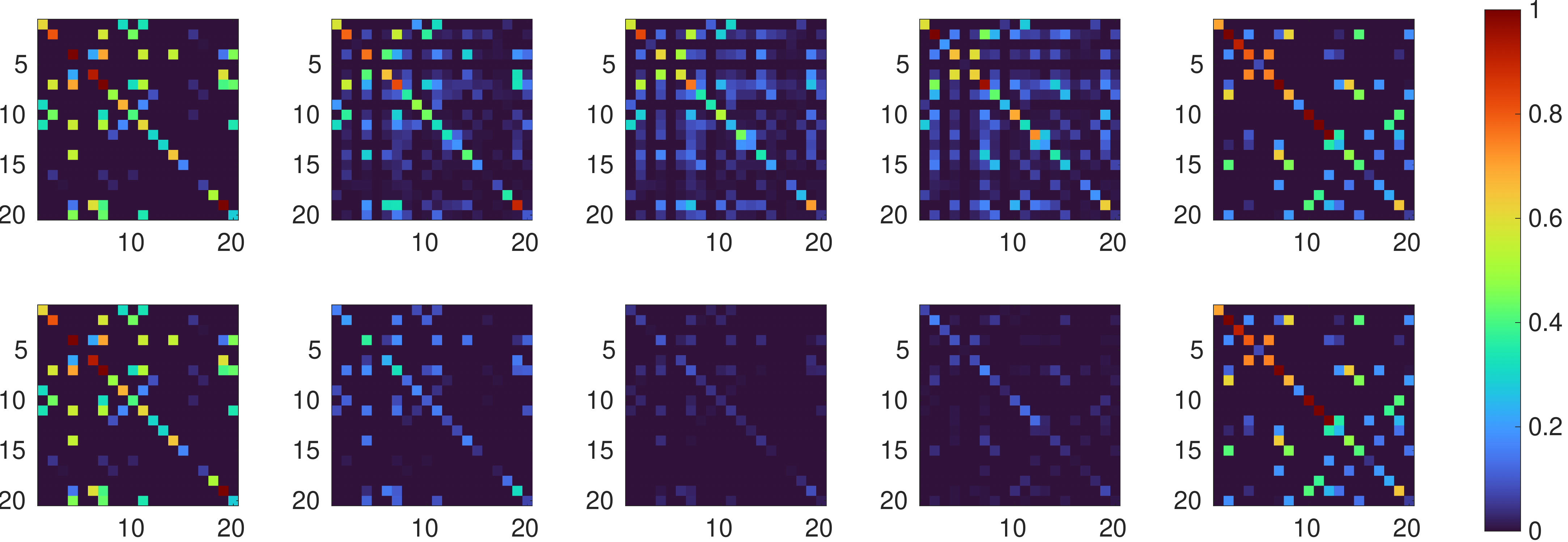}
    \caption{Points along the Riemannian (top row) and Thompson (bottom row) geodesic interpolations of a pair of $20\times 20$ SPD matrices with 68 non-zero entries. The matrices represent equidistant points along the geodesics as measured by the corresponding metric. Each pixel is colored according to the value of the corresponding matrix element. We observe that in the Riemannian case, most of the matrix elements along the interpolation are non-zero.}
    \label{fig:sparsity}
\end{figure}

\section{Inductive mean of SPD matrices based on Thompson geometry}
\label{sec:inductive mean}

A crucial step in developing a scalable computational framework for performing analysis and statistics on SPD-valued data using extreme generalized eigenvalues is to provide a suitable definition for the mean of a collection of $k$ SPD matrices whose computation can be based primarily on finding a sequence of extreme generalized eigenvalues. In this section, we introduce such a notion for any finite collection of SPD matrices through an iterative algorithm based on Thompson geodesics and prove that it yields a well-defined and unique point in each case. Furthermore, we highlight and prove a number of desirable properties that are satisfied by this novel inductive mean in \cref{properties_section}.

Specifically, given any finite ordered set $\mathcal{P}=(Y_1,\cdots,Y_k)\subset\mathbb{S}^n_{++}$, we generate a sequence of SPD matrices $(X_i)_{i\geq 1}$ from an arbitrary initialization $X_1\in \mathbb{S}^n_{++}$ according to \cref{alg:inductive mean}. We will then prove that any sequence generated by this algorithm converges to a point $X^*$ that is independent of the choice of initialization $X_1$ and the ordering of the $Y_j$, and thus can be viewed as a mean of the set of points $\{Y_j\}$.

\begin{algorithm}
\caption{Generate inductive sequence of SPD matrices $(X_i)_{i\geq 1}$ from an initial point $X_1$ and the finite ordered set $\mathcal{P}=(Y_1,\cdots,Y_k)\subset\mathbb{S}^n_{++}$}
\label{alg:inductive mean}
\begin{algorithmic}[1]
\FOR {$i\geq 1$} 
\STATE\label{line3}{Set $j \equiv i \mod k$ for $1\leq j \leq k$}
\STATE{Define $X_{i+1} = X_i*_{\frac{1}{i+1}}Y_j$}
\ENDFOR
\RETURN $(X_1,X_2,X_3,\cdots)$.
\end{algorithmic} 
\end{algorithm}

\subsection{Mathematical preliminaries}

We begin by presenting a number of technical lemmas that are used in the proof of our main theorem. First note that the Thompson geodesic \cref{Nussbaum geodesic} in $\mathbb{S}^n_{++}$ can be written as
\begin{equation}\label{geodesic_eq}
X*_tY = \varphi_{\alpha\beta}(t)Y+\psi_{\alpha\beta}(t)X,
\end{equation}
where
\begin{equation*}
\varphi_{\alpha\beta}(t) =
\begin{cases}
    \frac{\beta^t-\alpha^t}{\beta-\alpha} & \text{ if } \beta> \alpha\\
    t\alpha^{t-1}  & \text{ if } \beta = \alpha 
\end{cases} \quad \quad \quad
\psi_{\alpha\beta}(t) =
\begin{cases}
    \frac{\beta\alpha^t-\alpha\beta^t}{\beta-\alpha} & \text{ if } \beta> \alpha\\
    (1-t)\alpha^t  & \text{ if } \beta = \alpha
\end{cases}
\end{equation*}
\begin{equation}\label{derivatives}
\frac{d\varphi_{\alpha\beta}}{dt}(0) =
\begin{cases}
    \frac{\log\beta-\log\alpha}{\beta-\alpha} & \text{ if } \beta> \alpha\\
    \frac{1}{\alpha} & \text{ if } \beta = \alpha
\end{cases}  \quad \quad \quad
\frac{d\psi_{\alpha\beta}}{dt}(0) =
\begin{cases}
    \frac{\beta\log\alpha-\alpha\log\beta}{\beta-\alpha} & \text{ if } \beta> \alpha\\
    \log\alpha-1 & \text{ if } \beta = \alpha
\end{cases}.
\end{equation}
for $\alpha = \lambda_{\min}(YX^{-1})$ and $\beta = \lambda_{\max}(YX^{-1})$.

\begin{lemma}[Geodesic consistency]\label{geodesic_consistency_lemma}
    For $X,Y\in\mathbb{S}^n_{++}$ and $s,t\in[0,1]$
    \begin{equation}\label{geodesic_consistency_eq1}
    X*_tY = Y*_{1-t}X,
    \end{equation}
    \begin{equation}\label{geodesic_consistency_eq2}
    X*_s(X*_t Y) = X*_{st}Y
    \end{equation}
    and
    \begin{equation}\label{geodesic_consistency_eq3}
    (X*_sY)*_t Y = X*_{s+t-st}Y.
    \end{equation}
\end{lemma}
\begin{proof}
    \cref{geodesic_consistency_eq1} follows from the observation that for $0< \alpha\leq\beta$,
    $$\varphi_{\alpha\beta}(t) = \psi_{\beta^{-1}\alpha^{-1}}(1-t).$$
    So writing $\alpha = \lambda_{\min}(YX^{-1})$, $\beta = \lambda_{\max}(YX^{-1})$ we have by (\ref{geodesic_eq})
    $$X*_tY = \varphi_{\alpha\beta}(t)Y+\psi_{\alpha\beta}(t)X = \varphi_{\beta^{-1}\alpha^{-1}}(1-t)X+\psi_{\beta^{-1}\alpha^{-1}}(1-t)Y = Y*_{1-t}X.$$
    For \cref{geodesic_consistency_eq2}, suppose $\beta > \alpha$. Then
    \begin{equation*}
    \begin{aligned}
    \lambda_{\max}((X*_t Y)X^{-1})&= \lambda_{\max}((\varphi_{\alpha\beta}(t)Y+\psi_{\alpha\beta}(t)X)X^{-1}) \\
    &=\varphi_{\alpha\beta}(t)\lambda_{\max}(YX^{-1}) +\psi_{\alpha\beta}(t)\\
    &=\frac{\beta^t-\alpha^t}{\beta-\alpha}\beta+\frac{\beta\alpha^t-\alpha\beta^t}{\beta-\alpha}\\
    &=\beta^t.
    \end{aligned}
    \end{equation*}
    Similarly,
    $$\lambda_{\min}((X*_t Y)X^{-1}) = \alpha^t.$$
    These also hold when $\beta=\alpha$, and the proof is easier. Now use \cref{geodesic_eq} substituting the variables appropriately, or alternatively use \cite[Equation 1.25]{Nussbaum1994}, to get \cref{geodesic_consistency_eq2}. For \cref{geodesic_consistency_eq3},
    $$(X*_sY)*_tY= Y*_{1-t}(Y*_{1-s}X)= Y*_{(1-t)(1-s)}X = X*_{s+t-st} Y$$
    by \cref{geodesic_consistency_eq1} and \cref{geodesic_consistency_eq2}.
\end{proof}

For $i\in \N$ and $p\in\Z_{\geq 0}$ define the maps $S_i:\mathbb{S}^n_{++}\to\mathbb{S}^n_{++}$ and $T_p:\mathbb{S}^n_{++}\to\mathbb{S}^n_{++}$ by
\begin{equation*}
S_i: X\mapsto X*_{\frac{1}{i+1}}Y_j,
\end{equation*}
where $0\leq j\leq k$ is such that $j \equiv i \mod k$, and
\begin{equation*}
T_p:    X\mapsto \left(\dots(X*_{\frac{1}{pk+2}}Y_1)*_{\frac{1}{pk+3}}\dots\right)*_{\frac{1}{(p+1)k+1}}Y_k.
\end{equation*}
So if we pick an initialization $X_1\in\mathbb{S}^n_{++}$ for the algorithm, we have
\begin{equation*}
    X_{i+1}= S_i(X_i) \quad \mathrm{and} \quad X_{(p+1)k+1} = T_p(X_{pk+1}).
\end{equation*}
We will later need the following observation.

\begin{lemma}\label{scalar_lemma}
    If $c>0$, $i\in\N$ and $p\in\Z_{\geq 0}$, then
    \begin{equation*}
        S_i(cX)= c^{\frac{i}{i+1}}S_i(X) \quad \mathrm{and} \quad T_p(cX) = c^{\frac{pk+1}{(p+1)k+1}}T_p(X).
    \end{equation*}
\end{lemma}
\begin{proof}
    Observe that for $0<\alpha\leq\beta$ and $c>0$,
    \begin{equation*}
    \varphi_{(\alpha/c)(\beta/c)}(t) = c^{1-t}\varphi_{\alpha\beta}(t) \quad \mathrm{and} \quad
    \psi_{(\alpha/c)(\beta/c)}(t) = c^{-t}\psi_{\alpha\beta}(t)
    \end{equation*}
    and use the expression \cref{geodesic_eq}.
\end{proof}

For $1\leq j\leq k$ and $X\in\mathbb{S}^n_{++}(n)$, let $\varphi^j_X =\varphi_{\alpha\beta}$ and $\psi^j_X =\psi_{\alpha\beta}$ where $\alpha = \lambda_{\min}(Y_jX^{-1})$ and $\beta = \lambda_{\max}(Y_jX^{-1})$. We will also write $m^j_X = \frac{d\varphi^j_X}{dt}(0)$ and $o^j_X = \frac{d\psi^j_X}{dt}(0)$. Note that by \cref{derivatives} we see that $m^j_X$ is always positive, while $o^j_X$ may be positive, negative or zero.

From now on we will write $\|\cdot\|$ for the Euclidean (i.e. Frobenius) norm on matrices. 

\begin{lemma}\label{m_lipschitz}
For $1\leq j\leq k$, the maps from $\mathbb{S}^n_{++}$ to $\mathbb{R}$ given by 
\begin{equation*}
X \mapsto \varphi^j_X\text{, } \quad X \mapsto \psi^j_X\text{, } \quad X \mapsto \frac{d^2\varphi^j_X}{dt^2} \text{, } \quad X \mapsto \frac{d^2\psi^j_X}{dt^2}
\end{equation*}
are continuous. Moreover, if $\K\subset \mathbb{S}^n_{++}(n)$ is a compact set, the maps
\begin{equation*}
X \mapsto m^j_X \text{, }  \quad X \mapsto o^j_X
\end{equation*}
from $\K$ to $\mathbb{R}$
are Lipschitz with respect to the metric induced by $\|\cdot\|$ on $\K$, and in particular they are continuous.
\end{lemma}
\begin{proof}
We will show that $X \mapsto m^j_X$ is locally Lipschitz. The proof for $X \to o^j_X$ is analogous and the continuity of the other maps can also be shown in a similar fashion. 

$X \mapsto m^j_X$ can be expressed as the composition of the five maps
\begin{gather*}
\K\xrightarrow{\text{ I }}\K\times\K\xrightarrow{\text{ II }}\K\times\K \xrightarrow{\text{ III }} \R_{>0}\times\R_{>0}\xrightarrow{\text{ IV }} \R_{>0}\times\R_{>0}\xrightarrow{\text{ V }}  \R_{>0}  \\
X \xmapsto{\text{ I }}  (X,X^{-1}) \xmapsto{\text{ II }} (Y_jX^{-1}, XY_j^{-1}) \\
\xmapsto{\text{ III }} (\lambda_{\max}(Y_jX^{-1}), \lambda_{\max}(XY_j^{-1})) \xmapsto{\text{ IV }} (\lambda_{\max}(Y_jX^{-1}), \lambda_{\min}(Y_jX^{-1})) \xmapsto{\text{ V }} m_X^j.
\end{gather*}
Let us consider whether each of these five maps is Lipschitz. 

\begin{enumerate}[label=\Roman{*}:]
    \item Inversion is a smooth operation on the invertible matrices and $\K$ is a compact set of invertible matrices, so I is Lipschitz.
    \item Matrix multiplication is a smooth operation on matrices, so II is Lipschitz.
    \item This map is Lipschitz (and in fact non-expansive) with respect to the matrix norm $\|\cdot\|$.
    \item This map is given by inversion of the second coordinate. This is not Lipschitz. However we could restrict the domain to a compact subset of $\R_{>0}\times\R_{>0}$, since the image of the continuous map $(\text{map I})\circ(\text{map II})\circ(\text{map III})$ is a compact set. Then IV is Lipschitz on this domain.
    \item Explicitly, this map takes the form
\begin{equation*}
(\alpha,\beta)\mapsto \begin{cases}
    \frac{\log \beta-\log \alpha}{\beta-\alpha} & \text{ if } \beta \neq \alpha\\
    \frac{1}{\alpha} & \text{ if } \beta = \alpha.
\end{cases}
\end{equation*}
\end{enumerate}

We do not need V to be Lipschitz, we only need it to be locally Lipschitz, and then restrict the domain to a compact set like we did for IV. For this we show that its partial derivatives exist and are continuous. For $\beta \neq \alpha$,
$$\frac{\partial}{\partial \beta}(\text{map V})(\alpha,\beta) = \frac{1-\alpha/\beta-\log\beta+\log\alpha}{(\beta-\alpha)^2}.$$
As $(\alpha,\beta) \to (\gamma,\gamma)$,  the above tends to $-\frac{1}{2\gamma^2}$. This can be shown by letting $\alpha = \gamma + ta$ and $\beta = \gamma +tb$ for some $b$ and $a$, letting $t\to 0$ and applying l'Hôpital's rule twice. Moreover we have
$$\frac{\partial}{\partial \beta}(\text{map V})(\gamma,\gamma) = -\frac{1}{2\gamma^2}.$$
So the $\frac{\partial}{\partial\beta}$ derivatives exist and are continuous. The argument for the $\frac{\partial}{\partial\alpha}$ derivatives is analogous. This shows V is $C^1$ and hence locally Lipschitz. 

Now $m^j_X$ is Lipschitz in $X\in\K$ since it is a composition of Lipschitz maps.
\end{proof}

Pick an initialization $X_1\in\mathbb{S}^n_{++}$ and let $\C = \conv\{X_1,Y_1,\dots,Y_k\}$, where $\conv$ is used to denote the Euclidean convex hull.
If $X\in\mathbb{S}^n_{++}$, $1\leq j\leq k$ and $t\in[0,1]$,
\begin{equation}\label{projective_eq}
X*_tY_j = \big(\varphi^j_X(t)+\psi^j_X(t)\big)\Big(\frac{\varphi^j_X(t)}{\varphi^j_X(t)+\psi^j_X(t)}Y_j + \frac{\psi^j_X(t)}{\varphi^j_X(t)+\psi^j_X(t)}X\Big).
\end{equation}
Write $\R_{>0}\cdot \C = \{cX: X\in\C,c>0\}$. Then \cref{projective_eq} tells us that for all $i\in\N$, $S_i$ maps $\C$ to $\R_{>0}\cdot\C$. So by \cref{scalar_lemma} it maps $\R_{>0}\cdot\C$ to itself. In particular $X_i \in \R_{>0}\cdot\C$ for all $i\in\N$. 

\begin{lemma}\label{lim_lemma}
    There is $X^*\in\R_{>0}\cdot\C$ such that
    \begin{equation}\label{limit_eq}
    m^k_{X^*}Y_k+\dots+m^1_{X^*}Y_1+\sum_{j=1}^k o^j_{X^*}X^*=0.
    \end{equation}
\end{lemma}

\begin{proof}
Consider the map $F:\C\to\C$ defined by
\begin{equation*}
F: X \mapsto \frac{m^k_X Y_k +\dots+m^1_X Y_1}{\sum_{j=1}^k m^j_X}.
\end{equation*}
This is a continuous map (by continuity of the $m^j_X$ in $X$, \cref{m_lipschitz}) from the convex compact set $\C$ to itself, so it has a fixed point $X^{**}$ by Brouwer's fixed point theorem \cite[Theorem 4.10]{agarwal_fixed_2001}. Thus we have
\begin{equation*}
  X^{**} = \frac{m^k_{X^{**}} Y_k +\dots+m^1_{X^{**}} Y_1}{\sum_{j=1}^k m^j_{X^{**}}}.  
\end{equation*}
Now try $X^* = cX^{**}$ for $c>0$ in (\ref{limit_eq}). Using the relations $m^j_{cX^{**}} = cm^j_{X^{**}}$ and $o^j_{cX^{**}} = o^j_{X^{**}} - \log c$ and solving for $c$ we get a solution
\begin{equation}
    X^* = \exp\bigg(\frac{\sum_{j=1}^km^j_{X^{**}}+\sum_{j=1}^ko^j_{X^{**}}}{k}\bigg)X^{**}.
\end{equation}
\end{proof}
From now on $X^*$ will denote a point satisfying the conditions of \cref{lim_lemma}.
\begin{remark} Eventually we will show that $X_i \to X^* \text{ as } i\to\infty$ (and thus that $X^*$ is uniquely defined). This can be understood intuitively: writing out \cref{limit_eq} in a more explicit notation we have
\begin{equation}\label{exp_limit_eq}
\frac{d\varphi^k_{X^*}}{dt}(0)Y_k+\dots+\frac{d\varphi^1_{X^*}}{dt}(0)Y_1+\Big(\frac{d\psi^k_{X^*}}{dt}(0)+\dots+\frac{d\psi^1_{X^*}}{dt}(0)\Big)X^*=0.
\end{equation}
(\ref{exp_limit_eq}) is, loosely speaking, the infinitesimal version (taking $p\to \infty$) of the equation
\begin{equation*}
    T_p(X)=X,
\end{equation*}
which characterises the fixed point(s) of $T_p$.
\end{remark}

We are now in a position to prove the following crucial lemma.
\begin{lemma}\label{fixed_pt_lemma}
    Let $\K\subset\mathbb{S}^n_{++}$ be a compact set with $X^*\in\K$. Then there exist $O,K >0$ such that for all $X\in\K$ and $p\in\Z_{\geq 0}$,
    \begin{equation}\label{expansion_eq}
    \|T_p(X)-X\| \leq \frac{K}{p}\|X-X^*\|+\frac{O}{p^2}.
    \end{equation}
    In particular,
    \begin{equation}\label{fixed_pt_eq}
    \|T_p(X^*)-X^*\| \leq \frac{O}{p^2}.
    \end{equation}
\end{lemma}
\begin{proof}
For $1\leq j\leq k$, we define recursively
$$\K_j = \{X*_tY_j: X\in\K_{j-1}, t\in [0,1]\}$$
where $\K_0 = \K$. The $\K_j$ are continuous images of compact sets since the $\varphi^j_X$ and $\psi^j_X$ are continuous in $X$ (\cref{m_lipschitz}). Then for $i\in\N$ and $1\leq j\leq k$ such that $j\equiv i \mod k$, if $X\in\K_{j-1}$,
$$S_i(X) = \varphi^j_X\Big(\frac{1}{i+1}\Big)Y_j + \psi^j_X\Big(\frac{1}{i+1}\Big)X$$
so taking a Taylor expansion to first order we get
\begin{equation}\label{taylor1_eq}
S_i(X) = \frac{m^j_X}{i+1}Y_j + \Big(1+\frac{o^j_X}{i+1}\Big)X + R_j(X,i).
\end{equation}
where $\|R_j(X,i)\|\leq \frac{M_j}{i^2}$  for some $M_j>0$ independent of $X\in\K_{j-1}$. This bound on $R_j$ is possible because $\frac{d^2\varphi^j_X}{dt^2}$ and $\frac{d^2\psi^j_X}{dt^2}$ are uniformly bounded for $X\in\K_{j-1}$ and $t\in[0,\frac{1}{i+1}]\subset [0,1]$ by continuity on these compact sets (\cref{m_lipschitz}). 

Now for $X\in\K$ and $p\in\N$
\begin{equation*}\label{taylork_eq}
\begin{aligned}
    &T_p(X) = \\
    &\frac{m^k_{S_{(p+1)k}(\dots S_{pk+1}(X)\dots)}}{(p+1)k+1}Y_k
    +\Big(1+\frac{o^k_{S_{(p+1)k}(\dots S_{pk+1}(X)\dots)}}{(p+1)k+1}\Big)\frac{m^{k-1}_{S_{(p+1)k-1}(\dots S_{pk+1}(X)\dots)}}{(p+1)k}Y_{k-1} \\
    &\qquad\qquad\qquad +\dots+ \Big(1+\frac{o^k_{S_{(p+1)k}(\dots S_{pk+1}(X)\dots)}}{(p+1)k+1}\Big)\dots\Big(1+\frac{o^1_X}{pk+2}\Big)X+R(X,p) \\
    &=\frac{m^k_{S_{(p+1)k}(\dots S_{pk+1}(X)\dots)}}{pk}Y_k+\dots+\frac{m^1_X}{pk}Y_1 \\ &\qquad\qquad\qquad\qquad\qquad+\Big(1+\frac{o^k_{S_{(p+1)k}(\dots S_{pk+1}(X)\dots)}+\dots+o^1_X}{pk}\Big)X+S(X,p) \\
    &=\frac{m^k_X}{pk}Y_k+\dots+\frac{m^{1}_X}{pk}Y_1+\bigg(1+\frac{\sum_{j=1}^k o^j_X}{pk}\bigg)X+T(X,p) \\
    &=\frac{m^k_X-m^k_{X^*}}{pk}Y_k+\dots+\frac{m^1_X-m^1_{X^*}}{pk}Y_1+\frac{\sum_{j=1}^k(o^j_X-o^j_{X^*})}{pk}X \\
    &\qquad\qquad\qquad\qquad\qquad +\frac{\sum_{j=1}^ko^j_{X^*}}{pk}(X-X^*)+X+T(X,p) \\
    &=X+U(X,p)
\end{aligned}
\end{equation*}
where $R(X,p) \leq\frac{M}{p^2}$, $S(X,p) \leq\frac{N}{p^2}$, $T(X,p) \leq\frac{O}{p^2}$ and $U(X,p) \leq \frac{K}{p}\|X-X^*\|+\frac{O}{p^2}$ for some $L,M,N,O,K>0$ independent of $X\in\C$. The bound on $R$ comes from the expansion \cref{taylor1_eq} applied $k$ times and using the fact that $m^j_{X'}$ and $o^j_{X'}$ are continuous in $X'$ (\cref{m_lipschitz}), so bounded on the compact set $\K_k$. This last observation also gives us the bound on $S$. The bound on $T$ uses the fact that $\|S_j\circ\dots\circ S_1(X)-X\|$ vanishes to order $\frac{1}{p}$ for $X\in\K$ since the $m^j_{X'}$ and $o^j_{X'}$ are bounded for $X'\in\K_k$. Then we use the fact that $m^j_{X'}$ and $o^j_{X'}$ are Lipschitz in $X'\in\K_k$ (\cref{m_lipschitz}). Finally, the bound on $U$ uses the fact that $m^j_{X'}$ and $o^j_{X'}$ are Lipschitz in $X'\in\K$ (\cref{m_lipschitz}) and bounded on that set. This proves the lemma.
\end{proof}
\begin{remark}
Using similar estimates as in the proof of \cref{fixed_pt_lemma}, we can show that there are $O',K'>0$ such that for $X\in\K$ and $p\in\Z_{\geq0}$,
\begin{equation}\label{bound}
\|T_p(X)-X^*\| \leq K'\|X-X^*\|+\frac{O'}{p^2}.
\end{equation}
However it is not clear whether $K'<1$. If not, \cref{bound} is not good enough to show that the point $X^*$ is attractive under our dynamics, so we will need to use more machinery involving the Hilbert projective metric.
\end{remark}

\subsection{Hilbert projective convergence}\label{hilbert_section}

Here we establish convergence of any sequence generated by \cref{alg:inductive mean} in Hilbert's projective geometry. Recall that Hilbert's projective metric $d_H$ takes the form \cref{Hilbert matrix} in $\mathbb{S}^n_{++}$ and satisfies $d_H(cX,c'X)= d_H(X,X')$ for any $X,X'\in\mathbb{S}^n_{++}$ and $c,c'>0$. Moreover, $d_H$ is a metric in the usual sense on the projective space (space of rays) $\mathbb{S}^n_{++}/\R_{>0}$ \cite[Proposition 2.1.1]{Lemmens2012}. To proceed further, we need to be able to translate our estimates in the Euclidean norm to the Hilbert projective metric. This is achieved by the following lemma.
\begin{lemma}\label{hilbert_lemma}
Let $\K\subset \mathbb{S}^n_{++}$ be a compact set. Then there is $C>0$ such that for $X,X'\in\K$
$$d_H(X,X') \leq C\|X-X'\|.$$
\end{lemma}
\begin{proof}
For $X,X'\in\K$,
$$d_H(X,X') = \log\Big(\frac{\lambda_{\max}(X'X^{-1})}{\lambda_{\min}(X'X^{-1})}\Big).$$
So $d_H$ is the composition of the five maps
\begin{gather*}
\K\times\K\xrightarrow{\text{ I }}\K\times\K\times\K\times\K\xrightarrow{\text{ II }}\K\times\K \xrightarrow{\text{ III }} \R_{> 0}\times\R_{>0}\xrightarrow{\text{ IV }} \R_{>0}\times\R_{>0}\xrightarrow{\text{ V }}  \R_{\geq 0} \\
(X,X') \xmapsto{\text{ I }} (X,X^{-1},X',X'^{-1}) \xmapsto{\text{ II }} (X'X^{-1},XX'^{-1}) \\
\xmapsto{\text{ III }} (\lambda_{\max}(X'X^{-1}), \lambda_{\max}(XX'^{-1})) \xmapsto{\text{ IV }} (\lambda_{\max}(X'X^{-1}), \lambda_{\min}(X'X^{-1})) \\ 
\xmapsto{\text{ V }} \log\Big(\frac{\lambda_{\max}(X'X^{-1})}{\lambda_{\min}(X'X^{-1})}\Big).
\end{gather*}
Then we can show $d_H$ is Lipschitz analogously to the proof of \cref{m_lipschitz}.
\end{proof}

Let $\K\subset \mathbb{S}^n_{++}$ be a compact set. By \cref{Nussbaum inequality} (\cite[Theorem 1.2]{Nussbaum2004}), we have for $X,X',Y\in\R_{>0}\cdot\K$ and $t\in[0,1]$,
\begin{equation}\label{contraction_eq}
d_H(X*_tY,X'*_tY)\leq \gamma_{1-t}(R)d_H(X,X')
\end{equation}
where
\begin{equation*}
\gamma_{1-t}(R)= \frac{1-e^{-R(1-t)}}{1-e^{-R}}
\end{equation*}
and
\begin{equation*}
R = \diam_{d_H}(\R_{>0}\cdot\K) = \diam_{d_H}(\K) = \sup\{d_H(X,X'): X,X'\in\K\}<\infty
\end{equation*}
since $d_H(cX,c'X') = d_H(X,X')$ for all $c,c'>0$ and $\K$ is compact. We immediately get the following lemma.
\begin{lemma}[Hilbert contractivity]\label{contractive_lemma}
Let $\K\subset\mathbb{S}^n_{++}$ be a compact set. If $X,X'\in\R_{>0}\cdot\K$, $i\in\N$ and $p\in\Z_{\geq 0}$
$$d_H(S_i(X),S_i(X')) \leq \gamma_{\frac{i}{i+1}}(R)d_H(X,X').$$
and so
$$d_H(T_p(X),T_p(X')) \leq \gamma_{\frac{(p+1)k}{(p+1)k+1}}(R)\dots\cdot\gamma_{\frac{pk+1}{pk+2}}(R)d_H(X,X')$$
where $R = \diam_{d_H}(\K)$.
\end{lemma}
Note that taking the tangent line at $-R$ of the function $x\to e^x$ and using that this function is convex we get $e^x \geq e^{-R}+(x+R)e^{-R}$. So
\begin{equation}\label{cvx_inequality}
\gamma_{1-t}(R) =\frac{1-e^{-R(1-t)}}{1-e^{-R}} \leq 1-\frac{Re^{-R}}{1-e^{-R}}t.
\end{equation}
Now the following observation will turn out to be useful:
\begin{equation}\label{gamma_product}
\prod_{i=1}^\infty\gamma_{\frac{i}{i+1}}(R) = 0.
\end{equation}
This holds because
$$0\leq \prod_{i=1}^\infty\gamma_{\frac{i}{i+1}}(R) \leq \prod_{i=1}^\infty\Big(1-\frac{Re^{-R}}{1-e^{-R}}\frac{1}{i+1}\Big)=0$$
where the second inequality holds by \cref{cvx_inequality} and the last identity holds by \cite[Corollary 2.2.3]{arfken_5_1985} and using the divergence of the harmonic series.
\begin{remark}
\cref{gamma_product} combined with \cref{contractive_lemma} tells us that, given any two initializations for the algorithm, the resulting sequences will come arbitrarily close together in the Hilbert projective metric. However, this is not enough to show convergence in this projective metric. The key to showing this will be to also use \cref{fixed_pt_lemma}, with the help of \cref{hilbert_lemma}.
\end{remark}
\begin{proposition}[Hilbert convergence]\label{proj_conv_thm}
Let $(X_i)_{i\geq 1}$ be any sequence generated by \cref{alg:inductive mean} and $X^*$ denote a point satisfying the conditions of \cref{lim_lemma}. Then, we have
$$d_H(X_{pk+1},X^*) \to 0 \text{ as } p \to \infty.$$
\end{proposition}
\begin{proof}
For $p\in\Z_{\geq 0}$ we have
\begin{equation}\label{sol_ineq}
\begin{aligned}
d_H(X_{(p+1)k+1},X^*) &\leq d_H(X_{(p+1)k+1},T_p(X^*))+d_H(T_p(X^*),X^*) \\
&\leq \gamma_{\frac{(p+1)k}{(p+1)k+1}}(R)\dots\cdot\gamma_{\frac{pk+1}{pk+2}}(R)d_H(X_{pk+1},X^*) + \frac{D}{p^2}
\end{aligned}
\end{equation}
for some $D>0$, where we used  \cref{contractive_lemma} for the first term and \cref{hilbert_lemma} followed by (\ref{fixed_pt_eq}) from \cref{fixed_pt_lemma} for the second term.  \cref{contractive_lemma} was applied with $\K =\C$ and \cref{hilbert_lemma} was applied with
\begin{equation*}
\K = \{(\dots(X^**_{t_1}Y_1)*_{t_2}\dots )*_{t_k}Y_k: t_1,\dots,t_k\in[0,1]\}
\end{equation*}
which is compact by continuity of the $\varphi^j_X$ and $\psi^j_X$ in $X$ (\cref{m_lipschitz}). 

Now applying \cref{sol_ineq} recursively we get
\begin{equation*}
\begin{aligned}
d_H(X_{pk+1},X^*) &\leq \bigg(\prod_{i=k+1}^{pk}\gamma_{\frac{i}{i+1}}(R)\bigg)d_H(X_{k+1},X^*)+\sum_{q=1}^{p-1}\bigg(\prod_{i=qk+1}^{(p-1)k}\gamma_{\frac{i}{i+1}}(R)\bigg)\frac{D}{q^2} \\
&\leq \bigg(\prod_{i=k+1}^{pk}\gamma_{\frac{i}{i+1}}(R)\bigg)d_H(X_{k+1},X^*)+\sum_{q=1}^{\infty}\bigg(\prod_{i=qk+1}^{(p-1)k}\gamma_{\frac{i}{i+1}}(R)\bigg)\frac{D}{q^2} \\
&\to 0 \text{ as } p \to \infty
\end{aligned}
\end{equation*}
using \cref{gamma_product} and that $\sum_{q=1}^{\infty}\frac{D}{q^2}<\infty$, where we interpret $\prod_{i=l}^m a_i$ as 1 when $m<l$.
\end{proof}

\subsection{Convergence}\label{cvgence_section}

Let
\begin{equation*}
\S = \{X\in\mathbb{S}^n_{++}: \|X\|=\|X^*\|\}.    
\end{equation*}
So $X^*\in\S$. Moreover, if $X\in\mathbb{S}^n_{++}$ there is a unique $c>0$ such that $c^{-1}X\in\S$. So we have the natural identification
\begin{equation*}
\S\times\R_{>0}\xrightarrow{\cong} \mathbb{S}^n_{++}\text{, } 
\quad
(\hat X,c) \mapsto c\hat X.
\end{equation*}
Write $(\hat X_i,c_i)_{i\geq 1}\subset\S\times\R_{>0}$ for the sequence corresponding to $(X_i)_{i\geq1}\subset\mathbb{S}^n_{++}$. By \cref{proj_conv_thm}, $(X_{pk+1})_{p\geq 0}$ tends to $X^*$ in the Hilbert projective metric, hence so does $(\hat X_{pk+1})_{p\geq 0}$. The Hilbert projective metric is a metric in the proper sense on the set $\S$ \cite[Proposition 2.1.1]{Lemmens2012}. Moreover, the topology it generates is the Euclidean topology \cite[Proposition 1.1]{Nussbaum1994}. Hence $(\hat X_{pk+1})_{p\geq 0}$ actually converges to $X^*$ in the Euclidean topology, and so in the Euclidean norm. Now note that $X_{pk+1}= c_{pk+1}\hat X_{pk+1}$ for all $p\in\Z_{\geq 0}$. So we need to show $(c_{pk+1})_{p\geq 0}$ converges. 

We will slightly abuse the notation and view $T_p$ as a map from $\S\times\R_{> 0}$ to itself. With this in mind, for $\hat X\in \S$ write $b_{\hat X,p}$ for the positive number corresponding to the second coordinate of $T_p(\hat X,1)$. We will need to analyse these.
\begin{lemma}\label{bbound_lemma}
    Let $\K\subset\S$ be a compact set with $X^*\in\K$. Then there is $I,L>0$ such that for all $\hat X\in \K$ and $p\in\Z_{\geq 0}$ 
    \begin{equation*}
        \|b_{\hat X,p}-1\| \leq \frac{L}{p}\|\hat X-X^*\|+\frac{I}{p^2}.
    \end{equation*}
\end{lemma}
\begin{proof}
This is an exercise in Euclidean geometry using \cref{fixed_pt_lemma}. Let $\hat X\in\K$ and $p\in\Z_{\geq 0}$. Write $T_p(\hat X,1) = (\hat X',b_{\hat X,p})$. Then let $x$, $y$ and $z$ be the lengths of the segments $\hat X'$ to $b_{\hat X,p}\hat X'$, $b_{\hat X,p}\hat X'$ to $\hat X$ and $\hat X$ to $\hat X'$ respectively. Furthermore, let $h$ be the distance from $\hat X$ to the line through $0$ and $\hat X'$, and $\alpha$ be the angle between this line and the line through $0$ and $\hat X$ (see \cref{Euclidean_fig}).
Then
\begin{equation}\label{eq1}
y \leq \frac{K}{p}\|\hat X-X^*\|+\frac{O}{p^2}
\end{equation}
for some $O,K>0$ independent of $\hat X\in \K$ by \cref{expansion_eq} from \cref{fixed_pt_lemma}. So
\begin{equation}\label{eq2}
\alpha = \arcsin\frac{h}{\|\hat X\|}\leq \arcsin\frac{y}{\|\hat X\|} \to 0 \text{ as } p\to\infty
\end{equation}
uniformly for $\hat X \in \S$. Now
\begin{equation}\label{eq3}
z = \frac{h}{\cos(\alpha/2)}\leq \frac{y}{\cos(\alpha/2)}.
\end{equation}
So
\begin{equation*}
x\leq y+z \leq \Big(1+\frac{1}{\cos(\alpha/2)}\Big)y \leq \frac{K'}{p}\|\hat X-X^*\|+\frac{O'}{p^2}    
\end{equation*}
by \cref{eq3}, \cref{eq2} and \cref{eq1} for some $O',K'>0$ independent of $\hat X\in\K$. Dividing by $\|\hat X'\|= \|X^*\|$, we get the lemma.
\end{proof}
\begin{figure}[H]
\centerline{\includegraphics[scale=.25]{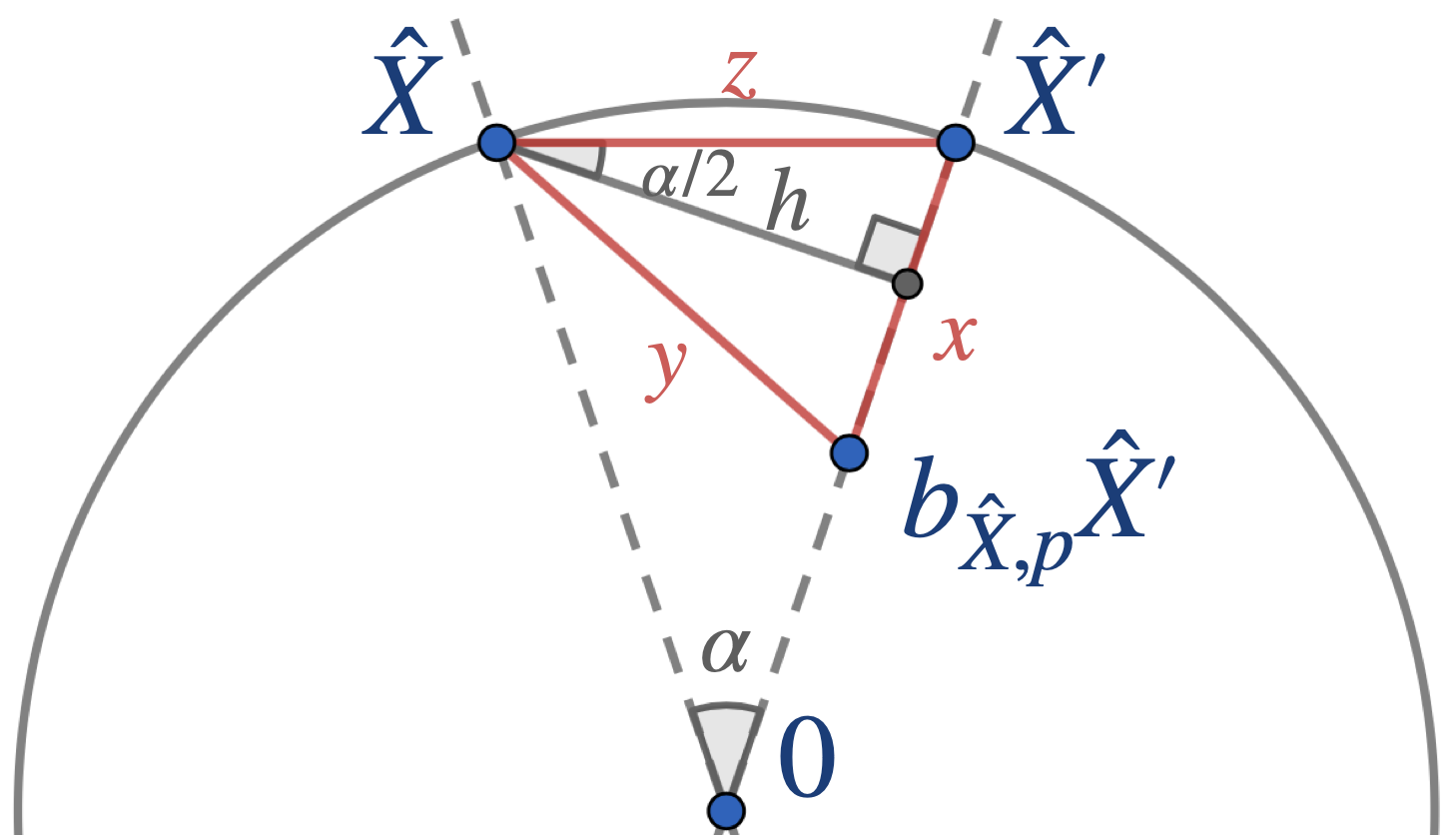}}
\caption{Sketch of the geometry involved in the proof of \cref{bbound_lemma}, Informally, we are trying to show $x$ is small. We know from \cref{fixed_pt_lemma} that $y$ is small, and we deduce that $z$ is small and thus that $x$ must be small.}
\label{Euclidean_fig} 
\end{figure}

Now we have all the necessary ingredients to prove convergence of $(c_{pk+1})_{p\geq 0}$.
\begin{proposition}[Radial convergence]\label{rad_conv_thm}
Let $(X_i)_{i\geq 1}$ be any sequence generated by \cref{alg:inductive mean} and $(c_i)_{i\geq 1}$ be the corresponding sequence in $\mathbb{R}_{>0}$ defined at the beginning of \cref{cvgence_section}.
Then, we have   
\begin{equation*}
 c_{pk+1} \to 1 \text{ as } p\to\infty.    
\end{equation*}
\end{proposition}
\begin{proof}
\cref{scalar_lemma} says that if $\hat X,\hat X'\in\S$ and $a,a'>0$ are such that $T_p(\hat X,a) = (\hat X', a')$, then for $c>0$,
\begin{equation}\label{scalar_transf_eq}
T_p(\hat X,ca) = (\hat X', c^{\frac{pk+1}{(p+1)k+1}}a').
\end{equation}
Then by (\ref{scalar_transf_eq})
\begin{equation}\label{recurrence_for_b}
c_{(p+1)k+1} = c_{pk+1}^{\frac{pk+1}{(p+1)k+1}}b_{\hat X_{pk+1},p}.
\end{equation}
Now since $\hat X_{pk+1} \to X^* \text{ as } p\to\infty$, \cref{bbound_lemma} applied to $\K =(\R_{>0}\cdot\C)\cap\S$ implies that for $\epsilon>0$ there is $r\in\Z_{\geq 0}$ such that
\begin{equation}\label{b_bound}
|b_{\hat X_{qk+1},q}-1|\leq \frac{\epsilon}{q}
\end{equation}
for all $q\geq r$. Also note that for $x\in\R$, $(1+x/q)^q \to e^x$ as $q \to \infty$. Thus, possibly by increasing $r$, we have
\begin{equation}\label{upper_bound}
\Big(1+\frac{\epsilon}{q}\Big)^q \leq \exp(2\epsilon)
\end{equation}
and
\begin{equation}\label{lower_bound}
\Big(1-\frac{\epsilon}{q}\Big)^q \geq \exp(-2\epsilon)
\end{equation}
for all $q\geq r$. Now applying \cref{recurrence_for_b} recursively we get for $p>r$
\begin{equation*}
c_{pk+1} = c_{rk+1}^{\frac{rk+1}{pk+1}}\prod_{q=r}^{p-1} b_{\hat X_{qk+1},q}^{\frac{(q+1)k+1}{pk+1}}    
\end{equation*}
so using \cref{b_bound} followed by \cref{upper_bound}
\begin{equation}\label{upper_bound2}
\begin{aligned}
c_{pk+1} &\leq c_{rk+1}^{\frac{rk+1}{pk+1}}\prod_{q=r}^{p-1}\Big(1+\frac{\epsilon}{q}\Big)^{\frac{(q+1)k+1}{pk+1}}\\
&\leq c_{rk+1}^{\frac{rk+1}{pk+1}}\prod_{q=r}^{p-1}\exp(2\epsilon)^{\frac{k+(k+1)/q}{pk+1}} \\
&= c_{rk+1}^{\frac{rk+1}{pk+1}}\exp(2\epsilon)^{\sum_{q=r}^{p-1}\frac{k+(k+1)/r}{pk+1}} \\
&\leq c_{rk+1}^{\frac{rk+1}{pk+1}}\exp(2\epsilon)^{1+2/r} 
\to \exp(2\epsilon)^{1+2/r} \text{ as } p\to\infty.
\end{aligned}
\end{equation}
We can similarly use \cref{b_bound} followed by \cref{lower_bound} to get
\begin{equation}\label{lower_bound2}
c_{pk+1}\geq c_{rk+1}^{\frac{rk+1}{pk+1}}\exp(-2\epsilon)^{1+2/r} \to \exp(-2\epsilon)^{1+2/r} \text{ as } p\to\infty.
\end{equation}
Since $\epsilon>0$ is arbitrary and $r\in\Z_{\geq 0}$ is arbitrarily large, we deduce from \cref{upper_bound2} and \cref{lower_bound2} that
$c_{pk+1} \to 1 \text{ as } p\to\infty$.    
\end{proof}
Finally, we are in a position to state and prove our main theorem.
\begin{theorem}[Convergence]\label{cvgence_thm} Let $(X_i)_{i\geq 1}$ denote any sequence generated by \cref{alg:inductive mean}. We have
    \begin{equation*}
    X_i \to X^* \text{ as } i\to\infty,
    \end{equation*}
    where $X^*$ is independent of the choice of initialization $X_1$. Moreover, $X^*$ is the unique solution in $\mathbb{S}^n_{++}$ to the equation
    \begin{equation}\label{limit_eq2}
    m^k_{X^*}Y_k+\dots+m^1_{X^*}Y_1+\sum_{j=1}^k o^j_{X^*}X^*=0,
    \end{equation}
    where $m^j_X = \frac{d\varphi^j_X}{dt}(0)$, $o^j_X = \frac{d\psi^j_X}{dt}(0)$, $\varphi^j_X =\varphi_{\alpha\beta}$, $\psi^j_X =\psi_{\alpha\beta}$, $\alpha = \lambda_{\min}(Y_jX^{-1})$, and $\beta = \lambda_{\max}(Y_jX^{-1})$.
\end{theorem}
\begin{proof}
We have shown with projective convergence (\cref{proj_conv_thm}) and radial convergence (\cref{rad_conv_thm}) that
\begin{equation*}
    X_{pk+1} = c_{pk+1}\hat X_{pk+1} \to X^* \text{ as } p\to\infty.
\end{equation*}
With the same proof we can show that for $1\leq j\leq k$, $(X_{pk+j})_{p\geq0}$ converges, and it does so to the same point $X^*$. So
$$X_i\to X^* \text{ as } i\to\infty.$$
To be precise, we have shown $X_i \to X^*$ for any $X^*\in \R_{>0}\cdot\C$ satisfying \cref{limit_eq2}. The proof can easily be extended to any $X^*\in\mathbb{S}^n_{++}$. Thus, by uniqueness of limits, the solution to \cref{limit_eq2} must be unique. Moreover we see that $X^*$ is independent of the choice of initialization $X_1$ by the symmetries of \cref{limit_eq2}, or alternatively by construction of $X^*$ in the proof of \cref{lim_lemma}.
\end{proof}

\begin{remark}
    \cref{cvgence_thm} holds more generally than just for the cone $\mathbb{S}^n_{++}$: given a closed convex cone in a finite dimensional real vector space $V$, take $\P= (y_1,\dots,y_k)$ in its interior. The cone is then almost Archimedean so we can still apply \cite[Theorem 2]{Nussbaum2004} in Section \ref{hilbert_section}, and it is also normal (\cite[Lemma 1.2.5]{Lemmens2012}) so we can also apply \cite[Proposition 1.1]{Nussbaum1994} in \cref{cvgence_section}. The only other parts of the proof that need generalizing are the proofs of \cref{m_lipschitz} and \cref{hilbert_lemma}. For this, we only need the additional assumption that, in the interior of this cone, the map $(x,y)\mapsto M(y/x)$
    is locally Lipschitz with respect to a norm on $V$. This is for example the case for the cone
    $\R^n_{+} = \{(x_i)_{i=1}^n:x_i\geq0,\;1\leq i\leq n\}$
    where
    $M(y/x) = \max\{y_i/x_i\}_{i=1}^n$.
\end{remark}

\subsection{Properties of the limit}\label{properties_section}

Since $X^*$ only depends on the $Y_j$ and not on $X_1$ we can write $M(Y_1,\dots,Y_k)= X^*$ and study how $M$ varies in terms of its arguments. We will see that $M$ satisfies a number of nice properties, and thus can be viewed as a mean of the $Y_j$. Some of these properties can be proved in several different ways. However, in a lot of cases, the characterisation of $M$ as being the unique solution to \cref{limit_eq2} grants us with some elegant proofs.
\begin{theorem}[Properties]\label{properties_thm}
Let $Y_1,\dots,Y_k \in \mathbb{S}^n_{++}$. 
\begin{enumerate}
\item For any $0\leq l\leq k$, 
\begin{equation*}
M(\underbrace{Y_1,\dots,Y_1}_{k-l},\underbrace{Y_2,\dots,Y_2}_{l}) = Y_1*_\frac{l}{k}Y_2,
\end{equation*}
and in particular $M(Y_1,Y_2) = Y_1*_\frac{1}{2}Y_2$.
\item Permutation invariance: 
\begin{equation*}
 M(Y_{\sigma(1)},\dots Y_{\sigma(k)}) = M(Y_1,\dots,Y_k) 
\end{equation*} for any permutation of $k$ elements $\sigma$.
\item Affine-equivariance: 
\begin{equation*}
 M(AY_1A^T,\dots, AY_kA^T) = AM(Y_1,\dots,Y_k)A^T   
\end{equation*} 
for any invertible matrix $A$.
\item Joint homogeneity: 
\begin{equation*}
M(c_1Y_1,\dots,c_kY_k) = (c_1\dots\cdot c_k)^{1/k}M(Y_1,\dots,Y_k)     
\end{equation*}
for any $c_1,\dots,c_k>0$.
\item The map $M: (\mathbb{S}^n_{++})^k\to\mathbb{S}^n_{++}$ is continuous.
\end{enumerate}
\end{theorem}
\begin{proof}
1. In this case it is perhaps easiest not to prove this property with \cref{limit_eq2} but instead to use \cref{alg:inductive mean} and \cref{geodesic_consistency_lemma}. Taking $X_1=Y_1*_\frac{l}{k}Y_2$ in the algorithm we have, using \cref{geodesic_consistency_lemma} recursively,
\begin{equation*}
    X_{pk+j} = 
    \begin{cases}
    Y_1*_{\frac{l}{k}\frac{pk+1}{pk+j}}Y_2 & \text{ if } 1\leq j \leq k-l \\
    Y_1*_{1-(1-\frac{l}{k})\frac{(p+1)k+1}{pk+j}}Y_2  & \text{ if } k-l+1 \leq j \leq k
    \end{cases}
\end{equation*}
for all $p\in\Z_{\geq0}$. In particular $X_{pk+1} = Y_1*_\frac{l}{k}Y_2$ for all $p\in\Z_{\geq0}$, and thus $X_i \to Y_1*_\frac{l}{k}Y_2$ as $i\to\infty$. 

2. Non-trivial from \cref{alg:inductive mean}, but follows immediately from the symmetries of \cref{limit_eq2}. 

3. Follows from \cref{alg:inductive mean} and \cref{affine_invariant_prop}. We give an alternative proof: writing $\tilde X^* = M(AY_1A^T,\dots, AY_kA^T)$ and $X^* = M(Y_1,\dots,Y_k)$, we have from \cref{limit_eq2}
\begin{equation*}
 m^k_{\tilde X^*}AY_kA^T+\dots+m^1_{\tilde X^*}AY_1A^T+\sum_{j=1}^k o^j_{\tilde X^*} \tilde X^*=0.   
\end{equation*}
Here we used that, for all $1\leq j \leq k$, $AY_jA^T(AXA^T)^{-1} = AY_jX^{-1}A^{-1}$ and $Y_jX^{-1}$ have the same (extreme) eigenvalues, and thus the coefficients $m^j_X$ and $o^j_X$ remain the same as in \cref{limit_eq2}. So we see that $\tilde X^* = AX^*A^T$ satisfies this equation.

4. Again, this is non-trivial from \cref{alg:inductive mean}. Writing $\tilde X^* = M(c_1Y_1,\dots,c_kY_k)$ and $X^* = M(Y_1,\dots,Y_k)$, we have from \cref{limit_eq2}
\begin{equation*}
 \tilde m^k_{\tilde X^*}c_kY_k+\dots+ \tilde m^1_{\tilde X^*}c_1Y_1+\sum_{j=1}^k  \tilde o^j_{\tilde X^*} \tilde X^*=0   
\end{equation*}
where $\tilde m^j_X = c_j^{-1}m^j_X$ and $\tilde o^j_X = o^j_X+\log c_j$ for all $1\leq j \leq k$ and all $X$. So
\begin{equation*}
m^k_{\tilde X^*}Y_k+\dots+ m^1_{\tilde X^*}Y_1+\sum_{j=1}^k  (o^j_{\tilde X^*}+\log c_j) \tilde X^*=0.    
\end{equation*}
Now it suffices to check that $\tilde X^* = (c_1\dots\cdot c_k)^{1/k}X^*$ satisfies this equation, using the relations $m^j_{cX} = cm^j_{X}$ and $o^j_{cX} = o^j_X -\log c$ for $c>0$. 

5. Consider the map $E: (\mathbb{S}^n_{++})^{k+1}\to\mathbb{S}^n_{++}$ given by
\begin{equation*}
       E: (Y_1,\dots,Y_k, X) \mapsto m^k_XY_k+\dots+m^1_XY_1+\sum_{j=1}^k o^j_X X.
\end{equation*}
We have by \cref{limit_eq2}
\begin{equation*}
E(Y_1,\dots,Y_k,M(Y_1,\dots,Y_k)) = 0    
\end{equation*}
for all $Y_1,\dots,Y_k\in\mathbb{S}^n_{++}$. One may try to apply the implicit function theorem to show that $M$ is continuous. However $E$ is not in general differentiable, only continuous (\cref{m_lipschitz}), so the implicit function theorem in its classical form cannot be applied. Instead, take $(Y_{1,i},\dots,Y_{k,i})_{i\geq 1}\subset (\mathbb{S}^n_{++})^{k+1}$ a convergent sequence, converging to $(Y_1,\dots,Y_k)$, say. Then writing $\tilde \C$ for the closure of the convex hull of the bounded set $\{(Y_{1,i},\dots,Y_{k,i}): i\in\N\}$, $\tilde \C$ is closed and bounded so compact. By the proof of \cref{lim_lemma}, $(M(Y_{1,i},\dots,Y_{k,i}))_{i\geq 1}\subset\K$ where
\begin{equation*}
\K = \bigg\{\exp\bigg(\frac{\sum_{j=1}^km^j_X+\sum_{j=1}^ko^j_X}{k}\bigg)X: X\in \tilde \C\bigg\},    
\end{equation*}
which is compact (\cref{m_lipschitz}). Thus, there is a convergent subsequence 
\begin{equation*}
(M(Y_{1,i_l},\dots,Y_{k,i_l}))_{l\geq 1}    
\end{equation*} converging to $M^*$, say. Then by continuity of $E$ we have
$E(Y_1,\dots,Y_k,M^*) = 0$.
But by uniqueness of the solution to \cref{limit_eq2}, $M(Y_1,\dots,Y_k) = M^*$. If $(M(Y_{1,i},\dots,Y_{k,i}))_{i\geq 1}$ did not converge to $M^*$, then it would have another convergent subsequence converging to a $\tilde M^*\neq M^*$. But then
$E(Y_1,\dots,Y_k,\tilde M^*) = 0$,
contradicting the uniqueness of the solution to \cref{limit_eq2}. So $M(Y_{1,i},\dots,Y_{k,i})\to M^*=M(Y_1,\dots,Y_k)$ as $i\to\infty$. So $M$ is continuous.
\end{proof}

\begin{corollary}[Structure preservation]\label{structure_preservation}
    If $S$ is a linear subspace of the space of real symmetric matrices and $\{Y_1\dots,Y_k\}\subset S\cap \mathbb{S}_{++}^n$, then $M(Y_1,\dots,Y_k)\in S$. 
\end{corollary}

\begin{remark}
    By \cref{structure_preservation}, the inductive Thompson mean defined by \cref{cvgence_thm} preserves many common matrix structures. In particular, the inductive Thompson mean of a collection of banded, Toeplitz, and Hankel matrices will be a unique banded, Toeplitz, and Hankel matrix, respectively. This is in contrast to the Riemannian (geometric) mean, which generally fails to preserve such structures. While there have been efforts to define structure-preserving geometric means by restricting the Riemannian barycenter computation to such subspaces, there is generally no guarantee that the result will be unique \cite{Bini2014Structured}.
\end{remark}

\begin{corollary}[Sparsity preservation I]\label{sparsity_preservation_1}
If $\{Y_1\dots,Y_k\}$ is a set of SPD matrices with the same sparsity pattern (i.e., with non-zero elements restricted to a common set of entries), then $M(Y_1,\dots,Y_k)$ has the same sparsity pattern.
\end{corollary}

\begin{corollary}[Sparsity preservation II]\label{sparsity_preservation_2}
If $\{Y_1\dots,Y_k\}\subset \mathbb{S}_{++}^n$ is a set of sparse SPD matrices with $k<<n^2$, then $M(Y_1,\dots,Y_k)$ is sparse.
\end{corollary}

\begin{remark}
We note that \cref{cvgence_thm} and \cref{properties_thm} continue to hold if we replace symmetric positive definite matrices with Hermitian positive definite matrices, with the only notable change being the need to use  conjugate transpose instead of transpose when defining properties such as affine-invariance.
\end{remark}

\section{Conclusions}
\label{sec:conclusions}

The Hilbert and Thompson metrics in the positive semidefinite cone provide a route to non-Euclidean geometries based on extreme generalized eigenvalue computations. We have seen that by focusing on a particular choice of geodesic of the Thompson metric with attractive computational properties, we can view $(\mathbb{S}^n_{++},d_T)$ as a semihyperbolic geodesic space. We have noted several interesting properties of this distinguished Thompson geodesic, including the preservation of sparsity. Significantly, we have defined an inductive mean of any finite collection of SPD matrices based on the computation of a sequence of extreme generalized eigenvalues. Furthermore, we have proved that this new mean exists and is unique for any given finite collection of SPD matrices. Finally, we have established several important properties that are satisfied by this mean, including permutation invariance, affine-equivariance, and joint homogeneity. We hope that these contributions will provide a foundation for a computationally scalable geometric statistical framework for the processing of large SPD-valued data.

% \section*{Acknowledgments}

\bibliographystyle{siamplain}
\bibliography{references}
\end{document}

%% file: ex_shared.tex
\usepackage{lipsum}
\usepackage{amsfonts}
\usepackage{graphicx}
\usepackage{epstopdf}
\usepackage{algorithmic}
\usepackage{mathtools}
\usepackage{enumerate}
\usepackage{enumitem}
\ifpdf
  \DeclareGraphicsExtensions{.eps,.pdf,.png,.jpg}
\else
  \DeclareGraphicsExtensions{.eps}
\fi

\newcommand{\N}{\mathbb N}
\newcommand{\Z}{\mathbb Z}
\newcommand{\R}{\mathbb R}
\newcommand{\C}{\mathbb C}
\renewcommand{\P}{\mathbb P}

\newcommand{\diam}{\operatorname{diam}}

\newcommand{\conv}{\operatorname{conv}}

\newcommand{\K}{\mathcal K} %%
\renewcommand{\C}{\mathcal C} %%
\renewcommand{\P}{\mathcal P} %%
\renewcommand{\S}{\mathcal S} %%

\newsiamremark{remark}{Remark}
\newsiamremark{hypothesis}{Hypothesis}
\crefname{hypothesis}{Hypothesis}{Hypotheses}
\newsiamthm{claim}{Claim}

\headers{Differential geometry with extreme eigenvalues}{C. Mostajeran, N. Da Costa, G. Van Goffrier, and R. Sepculchre}

\title{Differential geometry with extreme eigenvalues in the positive semidefinite cone\thanks{Submitted to the editors DATE.
\funding{C.M. was supported by a Presidential Postdoctoral Fellowship at Nanyang Technological University (NTU Singapore) and an Early Career Research Fellowship at the University of Cambridge.
G.V.G. was supported by the UCL Centre for Doctoral Training in Data Intensive Science funded by STFC, and by an Overseas Research Scholarship from UCL. The research leading to these results has also received funding from the European Research Council under the
Advanced ERC Grant Agreement SpikyControl n.101054323.}}}

\author{Cyrus Mostajeran\thanks{School of Physical and Mathematical Sciences, Nanyang Technological University, Singapore.}
\and Nathaël Da Costa\footnotemark[2]
\and Graham Van Goffrier\thanks{Department of Physics and Astronomy, University College London, London, UK.}
\and Rodolphe Sepulchre\thanks{Department of Engineering, University of Cambridge, UK \& Department of Electiral Engineering, KU Leuven, Belgium.}}

\usepackage{amsopn}